\documentclass[12pt]{amsart}

\usepackage[utf8]{inputenc}

\usepackage{amscd}
\usepackage{amsmath}
\usepackage{amssymb}
\usepackage{amsthm}
\usepackage{arydshln}
\usepackage{fancyhdr}
\usepackage{verbatim}
\usepackage[usenames, dvipsnames]{color}
\usepackage{tikz}
\usetikzlibrary{snakes}
\usepackage[all]{xy}
\usepackage[colorlinks=true, linkcolor=blue, citecolor=blue, pagebackref=true]{hyperref}
\usepackage[pagebackref=true]{hyperref}
\usepackage{mathtools}

\usepackage{mathrsfs}
\renewcommand{\epsilon}{\varepsilon}
\renewcommand{\phi}{\varphi}

\newcommand{\N}{\mathbb{N}}
\newcommand{\R}{\mathbb{R}}
\newcommand{\Z}{\mathbb{Z}}

\newcommand{\W}{\mathcal W}
\newcommand{\PP}{\mathcal P}
\newcommand{\E}{\mathcal E}
\newcommand{\calS}{\mathcal S}

\newcommand{\B}{\mathcal B}
\newcommand{\LL}{\mathcal L}
\newcommand{\both}{\text{I/II}}
\newcommand{\I}{\textrm{I}}
\newcommand{\II}{\textrm{II}}

\newcommand{\reg}{\textrm{reg}}
\newcommand{\sing}{\textrm{sing}}

 \newtheorem{theorem}{Theorem}
 \newtheorem{lemma}[theorem]{Lemma}
 \newtheorem{corollary}[theorem]{Corollary}

 \numberwithin{equation}{section}
 \numberwithin{theorem}{section}

\theoremstyle{definition}
\newtheorem{remark}[theorem]{Remark}

\providecommand{\abs}[1]{\lvert#1\rvert}

\providecommand{\norm}[1]{\lVert#1\rVert}

\providecommand{\inner}[2]{\langle#1,#2\rangle}

\title[Typical statistics]{\textbf{Statistics of patterns in typical cut and project sets}}
\author[Haynes, Julien, Koivusalo, Walton]{Alan Haynes, Antoine Julien, Henna Koivusalo,\\ James Walton}
\thanks{Research of AH, JW, HK supported by EPSRC grants L001462, J00149X, M023540. HK by Osk. Huttunen foundation.}
\date{\today}

\begin{document}

\begin{abstract}
In this article pattern statistics of typical cubical cut and project sets are studied. We give estimates for the rate of convergence of appearances of patches to their asymptotic frequencies. We also give bounds for repetitivity and repulsivity functions. The proofs use ideas and tools developed in discrepancy theory.
\end{abstract}

\maketitle


\section{Introduction}

\subsection{Overview}

Cut and project sets, or model sets, are an important class of point sets which are not periodic, yet are extremely regular. They were introduced by Meyer in the framework of harmonic analysis, as a generalisation of lattices, see~\cite{Moo97}. 

The most acclaimed application of these points sets to this date has been in crystallography. In the early 1980s, Nobel Prize laureate Dan Shechtman~\cite{SchBleGraCah84} discovered a material for which the diffraction pattern had both sharp peaks (a feature of order), and five-fold symmetry (an obstruction to periodicity). Cut and project point sets, such as the Penrose patterns, provide instructive models of these materials. Properties of cut and project sets are also being studied in connection with signal sampling and reconstruction~\cite{MatMey10}.

The idea of the cut and project method is the following. Starting with a lattice, such as the standard integer lattice $\Z^k \subset \R^k$, pick a $d$-dimensional subspace $E$ (the \emph{physical space}), and cut a slice $E+\W$ from the lattice, for some \emph{window} $\W$. The cut and project set associated with this data is the projection of $\Z^k \cap (E+\W)$ to $E$. With appropriate irrationality conditions on $E$, the resulting point set is not periodic in $E$, yet for sensible choices of $\W$ it inherits some of the regularity from the original lattice. The data of a cut and project set consists of the following parameters: the subspace $E$, the window $\W$, and the projection $\pi$ onto $E$ (determined by a choice of a complementary subspace $F_\pi$ to $E$). It is natural to ask how the properties of the cut and project set change when these parameters vary.

We will consider three functions when investigating how ordered a cut and project set is. Firstly, the complexity function: how many patches of a given size are there? This question has been satisfactorily answered in \cite{Jul12} for polytopal windows. Secondly, the repetitivity function: given a size $r$, how far does one need to look from any given point of the point set before all `legal' patches of size $r$ can be found? Finally, the discrepancy: a patch of size $r$ has an expected frequency; what is the difference between the expected and the actual number of times this patch occurs in a large region? While these quantities have already been studied for isolated examples of cut and project sets, their `typical' behaviour (in a measure theoretic sense to be made precise below) is largely unexplored.

As we will see, the importance of $F_\pi$ is only marginal as long as a few degenerate cases are ruled out.
Viewing $E$ as the graph of a linear transformation $L: \R^d \rightarrow \R^{k-d}$, we look for properties that hold for almost every choice of entries in the matrix associated to $L$. We restrict our attention to \emph{cubical} windows; that is, we assume that $\W$ is a $(k-d)$-dimensional face of the unit cube. As a point of reference, when $\W$ is the unit cube of $\R^k$, the cut and project data is usually called \emph{canonical}.

For some choices of dimensions $k$ and $d$, the corresponding cut and project sets are already well studied. When $d=1$ and $k=2$, cubical and canonical cut and project sets (in the case when the line $E$ has irrational slope) are known as Sturmian sequencies~\cite{Fog02, MH38}. These sets provide a partition of the line $E$ to two types of intervals, short and long, and can be symbolically coded by an element in $\{a,b\}^\Z$. It is known that many properties of Sturmian sequences can be determined explicitly by Diophantine properties of the slope $\theta$ of $E$, using in particular the continued fraction expansion of $\theta$. More generally, when $d=1$, cut and project sets are closely related to cubic billiard sequences. Some attempts have also been made to code $d$-dimensional cut and project sets, with $d>1$, as $\Z^d$-subshifts (sometimes under the name of ``discrete planes''), see~\cite{BV00} and~\cite[Chap.~10]{FHK02}.

When it comes to repetitivity, most of the known results have to do with extreme cases, rather than generic. For example, the lowest possible growth rate for the repetitivity function of an aperiodic point set is linear~\cite{LP02}. When the repetitivity is indeed bounded above by a linear function, the point set is called \emph{linearly repetitive}, or LR.
In a previous paper~\cite{HaynKoivWalt2015a}, a criterion was given for cut and project sets to be linearly repetitive. In particular, it was shown that almost no cubical or canonical cut and project set is LR. However, the almost everywhere behaviour of the repetitivity function has not been studied in general.

In the particular case of Sturmian sequences, Hedlund and Morse proved in their original series of papers~\cite{MH38,MH40} that for almost all $\theta$ and for all $\varepsilon > 0$, the repetitivity function $M(r)$ of the Sturmian sequence of slope $\theta$ is bounded above by $r\log(r)^{1+\varepsilon}$. It would appear that Sturmian sequences are not generically LR, but are fairly close to this optimal behaviour. In the present paper, we establish a higher-dimensional analogue of the Hedlund and Morse result. Theorem~\ref{main:transference} states that, given $d < k$, there exist $\delta > 0$ (which we determine explicitly in terms of $d$ and $k$) such that for Lebesgue almost any parameters in the cut and project method, the repetitivity function is bounded above by $r^{k-d} (\log r)^{\delta}$.

Discrepancy estimates for the frequencies of aperiodic Delone sets or aperiodic tilings have been studied by many authors, but the results are usually stated for the most regular classes of point sets: linearly repetitive or self-similar, see ~\cite{Ada04, ACG11, BufSol13, LP02, Sad11}. In Theorem~\ref{main:discrepancy} we give a discrepancy estimate for generic cubical cut and project sets. In short, the typical convergence to ergodic averages is very fast when looking at patches of a certain form. We also give a weaker bound which applies to more general patches which, for $d \geq 2$, is asymptotically the best possible over general search regions.

\subsection{Statement of results}
A (Euclidean) {\bf cut and project scheme} consists of the following data:
\begin{itemize}
	\item A {\bf total space} $\R^k$.
	\item A linear subspace $E \subset \R^k$ of dimension $d$ with $0 < d < k$, called the {\bf physical space}.
	\item A linear subspace $F_\pi \subset \R^k$, complementary to $E$ in $\R^k$, called the {\bf internal space}.
	\item A subset $\W_\pi \subset F_\pi$ called the {\bf window}.
\end{itemize}
The decomposition $\R^k = E+F_\pi$ defines the projections $\pi$ and $\pi^*$ onto $E$ and $F_\pi$, respectively. The {\bf slice} is defined as $\mathcal{S} \coloneqq \W_\pi + E$. Given $s \in \R^k$, we define the {\bf cut and project set}
\[
Y_s \coloneqq \pi(\mathcal{S} \cap (\Z^k+s)).
\]

In this paper, we restrict our attention to {\bf cubical} cut and project sets, which means that the window is given by $\pi^*([0,1]^{k - d}\times \{0\})$. The physical space will be assumed to be {\bf totally irrational}, which means that $E+\Z^k$ is dense in $\R^k$ (or, equivalently, that $\pi^*(\Z^k)$ is dense in $F_\pi$). We also adopt the conventional assumption that $\pi$ is injective on $\Z^k$. With these restrictions, the patch statistics of interest in this paper will be wholly dependent on the choice of physical space $E$, and in particular on its Diophantine properties. The physical space $E$ will always be given as the graph of a linear map $L : \R^d \to \R^{k-d}$, that is
\[
E = \{(x,L(x)) \mid x \in \R^d\}.
\]
It might be necessary to permute the indexing of the coordinate axes in order to write $E$ in this manner, but there is no loss in generality in doing so. We write $L_i(x) \coloneqq L(x)_i = \sum_{j=1}^d \alpha_{ij} x_j$ and use the coefficients $(\alpha_{ij}) \in \R^{d(k-d)}$ to parametrise the choice of physical space.

We say that $s \in \R^k$ and its corresponding cut and project set $Y_s$ are {\bf regular} if $\partial \mathcal{S} \cap (\Z^k+s) = \emptyset$.
Because of their special repetitivity properties, it is usual to restrict attention to regular cut and project sets.
The cut and project sets of this family all have the same finite sub-patches, so in what follows, given a cut and project scheme, there will be no loss in generality in assuming that some regular $Y=Y_s$ has been chosen.

Given $y \in Y$ and $r \in \mathbb{R}_+$, denote by $P(y,r)$ the {\bf $r$-patch at $y$}, which we think of as the pattern of points of $Y$ within distance $r$ of $y$. The precise definition of $P(y,r)$ will be given in Section~\ref{sec:toolbox} although, as shall be discussed in detail in Section \ref{sec:Intrinsically defined patches and search regions}, most of our results will not be dependent upon the particular choice of notion of $r$-patch. We consider two $r$-patches $P(y_1,r)$ and $P(y_2,r)$ to be equivalent, and write $P(y_1,r) \simeq P(y_2,r)$, if
\[
P(y_1,r)-y_1 = P(y_2,r)-y_2.
\]
Write $\PP(y,r)$ for the equivalence class of an $r$-patch $P(y,r)$ and, for $y \in Y$, let $\tilde{y}$ denote the unique element of $\Z^k + s$ with $\pi(\tilde{y})=y$. For an equivalence class $\PP$, $y \in Y$ and $R>0$, define
\[
\xi_\PP(y,R) \coloneqq \frac{\#\{y' \in Y \mid \PP(y',r) = \PP \text{ and } \tilde{y}'-\tilde{y} \in [-R,R]^d \times \R^{k-d}\}}{\#\{y' \in Y \mid \tilde{y}'-\tilde{y} \in [-R,R]^d \times \R^{k-d}\}}
\]
In other words, $\xi_\PP(y,R)$ is the ratio of the number of occurrences of $\PP$ in $Y$ in a box of size $R$ around $y$ (with some points of $Y$ carefully chosen in or out of the box near its boundary) relative to the total number of points of $Y$ in this box.

The {\bf frequency} of $\PP$ is defined to be
\[
\xi_\PP \coloneqq \lim_{R \to \infty} \xi_\PP(y,R).
\]
The above limit is always well-defined and does not depend on the choice of $s \in \R^k$ or $y \in Y$. Our first theorem concerns the typical rate of convergence of the estimates $\xi_\PP(y,R)$ to the asymptotic patch frequencies $\xi_\PP$. Its proof is given in Section \ref{sec:discrepancy proof}.
\begin{theorem}\label{main:discrepancy}
Fix $\epsilon >0$. Then for almost all linear maps $L \colon \R^d\to \R^{k-d}$, for the corresponding cubical cut and project sets we have the bound
\[
|\xi_\PP(y,R) - \xi_\PP|\le C\cdot \frac {(\log R)^{k + \epsilon}}{R^d},
\]
for all $R \geq 1$ for all equivalence classes of patches $\PP$. The constant $C>0$ depends on $L, \epsilon$ and $\pi$.
\end{theorem}
The above theorem shows that the discrepancy of patches in typical cubical cut and project sets is remarkably low. We emphasize that the constant $C$ in the above result does not depend on the particular equivalence class of patch $\PP$ in question. Something that allows this discrepancy to be so low is that, as indicated above, which points near the boundary of an $r$-patch of $\PP$ are to be considered in or out, and which points near the boundary of an $R$-neighbourhood of a given point $Y$ are to be counted, are judiciously chosen extrinsically with respect to the cut and project scheme. One may ask how these estimates change if the shapes of patches or search regions are altered; variants such as this are discussed in Section \ref{sec:Intrinsically defined patches and search regions}. The above estimate in fact allows one to give good bounds for more general patch types, over more general search regions.

Given $\phi : \mathbb{R}_+ \to \mathbb{R}_+$, we shall say that $Y$ is {\bf $\phi$-repetitive} if, for sufficiently large $r$, for any equivalence class of $r$-patch $\PP$ and any $y \in Y$, there exists some $y'$ in the ball $B(y,\phi(r))$ with $\PP(y',r) = \PP$. In other words, for every $r$-patch that occurs \emph{somewhere} in the pattern one may find a translate of it within $\phi(r)$ of \emph{any} point of the pattern. Every cut and project set $Y$ is repetitive, which means that it is $\phi$-repetitive for some $\phi$. In \cite{HaynKoivWalt2015a} the question of which cubical cut and project sets are {\bf linearly repetitive} ({\bf LR}) was investigated. LR means that $Y$ is $\phi$-repetitive with $\phi(r) = Cr$ for some $C>0$. The results of \cite{HaynKoivWalt2015a} show that this property is rare, in that a typical cut and project set is not LR (and for some choices of $d$ and $k$ there are no non-trivial examples of LR cut and project sets). The following theorem, proved in Section \ref{sec:repetitivity}, gives a bound for the repetitivity function of a typical cut and project set.
\begin{theorem}\label{main:transference}
Fix $c, \epsilon>0$. For Lebesgue almost all linear maps $L:\R^d\to \R^{k-d}$, the corresponding cubical cut and project sets are $\phi$-repetitive for
\[
\phi(r) \ge Cr^{k-d}(\log r)^{\frac{2k-1}{d} - 1 + \epsilon}
\]
but are not $\phi$-repetitive for
\[
\phi(r) \le cr^{k-d} (\log r)^{1/d}.
\]
The constant $C$ depends on $L, \epsilon$ and $\pi$.
\end{theorem}
Repetitivity measures the largest gap between consecutive appearances of the same $r$-patch. Another measure of the regularity of patterns, which in some sense is dual to this, is repulsivity: what is the \emph{smallest} gap between consecutive appearances of the same $r$-patch? We say that $Y$ is {\bf $\phi$-repulsive} for some $\phi : \R_+ \rightarrow \R_+$ if, for sufficiently large $r$, whenever $\PP(y,r) = \PP(y',r)$ for distinct $y,y' \in Y$ then $d(y,y') > \phi(r)$. So if $Y$ is $\phi$-repulsive for a `large' function $\phi$, then any two occurrences of an $r$-patch of $Y$ are forced to be far apart, relative to $r$. It is common in the literature to refer to $Y$ as {\bf repulsive} if $Y$ is $\phi$-repulsive with $\phi(r)=cr$ for some $c > 0$. It is known that $Y$ being repulsive is a necessary condition for $Y$ to be LR (see \cite[Lemma 2.4]{Sol98}), and for $d=k-d=1$, $Y$ being LR is equivalent to $Y$ being repulsive.
The following theorem, proved in Section \ref{sec:repetitivity}, studies repulsivity of typical cut and project sets.
\begin{theorem} \label{thm:repulsive}
Fix $c,\epsilon > 0$. For Lebesgue almost all linear maps $L : \R^d \rightarrow \R^{k-d}$, the corresponding cubical cut and project sets are not $\phi$-repulsive for
\[
\phi(r) \ge \frac{cr^{k-d}}{(\log r)^{1/d}}
\]
but are $\phi$-repulsive for
\[
\phi(r) \le \frac{cr^{k-d}}{(\log r)^{(1/d)+k-d +\epsilon}}.
\]
\end{theorem}
In Section \ref{sec:toolbox} we give a precise definition of a patch and gather together the lemmas and observations on cut and project sets and Diophantine approximation that will be necessary for our proofs. Sections \ref{sec:discrepancy proof} and \ref{sec:repetitivity} contain the proofs of the main theorems. In Section \ref{sec:Intrinsically defined patches and search regions} we consider other types of $r$-patches and discrepancy counts.

\subsection{Notation}
For $x \in \R$, ${\bf x} \in \R^m$
\begin{itemize}
\item $\norm{x}$ distance to the nearest integer.
\item $\abs{x}$ absolute value.
\item $\lfloor x \rfloor$ integer part.
\item $\abs{{\bf x}}=\max_{i=1, \dots, m}\abs{x_i}$.
\item $\norm{{\bf x}}=\max_{i=1, \dots, m}\norm{x_i}$.
\end{itemize}

We use the symbols $\ll$ and $\gg$ for the standard Vinogradov notation. When using this notation, if the implied constants depend on the variables involved, unless otherwise noted, this will be indicated by the use of subindices. For a measurable set $A \subseteq \R^m$, $|A|$ denotes the Lebesgue measure of $A$.

\section{Toolbox}\label{sec:toolbox}
\subsection{Patches in cut and project sets}

For $A \subseteq \R^m$, we define $A_\Z \coloneqq A \cap \Z^n$. We let $C(r) \coloneqq ([-r,r]^d \times \R^{k-d})_{\Z}$, that is, $C(r)$ is the cylinder of lattice points of the total space $\R^k$ whose first $d$ coordinates lie in the box $[-r,r]^d$. In addition to the internal space $F_\pi$ we will often be working with a {\bf reference space}, defined as $F_\rho \coloneqq \{0\}^d \times \R^{k-d}$. The decomposition $\R^k = E+F_\rho$ defines the projections $\rho$ and $\rho^*$ onto $E$ and $F_\rho$, respectively. Let $\W=\rho^*(\mathcal S)$, which we shall also refer to as the {\bf window}. It will often be necessary to consider linear maps as maps to the torus, and in these instances we use the corresponding calligraphic letters; for a linear map $L:\R^m\to \R^\ell$, say, we denote by $\mathcal L$ the mapping $L \mod 1:\R^m\to \R^\ell/\Z^\ell$.

Given $y \in Y$ and $r \in \R_+$, we define the {\bf $r$-patch at $y$} to be
\[
P(y,r) \coloneqq \{y' \in Y \mid \tilde{y}'-\tilde{y} \in C(r)\}
\]
where, for $y \in Y$, $\tilde{y}$ is the unique element of $\Z^k+s$ for which $\pi(\tilde{y}) = y$. So $P(y,r)$ consists of the points of $Y$ which are projections of points whose first $d$ coordinates differ from $\tilde{y}$ by at most $r$. While there are more geometrically intuitive notions of $r$-patches from the perspective of $Y$ as a Delone set, this definition is natural in terms of the cut and project scheme and will be technically simple to work with. As it turns out, we will show in Section \ref{sec:Intrinsically defined patches and search regions} that many of our results are not dependent upon the precise notion of $r$-patch used.

With the above notation, the estimate $\xi_\PP(y,R)$ of the frequency of equivalence class of $r$-patch $\PP$ at a point $y \in Y$ to distance $R$ is given as
\[
\xi_\PP(y,R) \coloneqq \frac{\#\{y' \in Y \mid \PP(y',r) = \PP \text{ and } \tilde{y}' - \tilde{y} \in C(R)\}}{\#\{y' \in Y \mid \tilde{y}' - \tilde{y} \in C(R)\}}.
\]
As with our definition of $r$-patch, the `search-region' about $y$ of points $y' \in Y$ with $\tilde{y}'-\tilde{y} \in C(R)$ has a somewhat extrinsic definition in terms of the cut and project scheme (note also that the number of such points is precisely $(2\lfloor R \rfloor + 1)^d$).
We prove results for their intrinsic counterparts in Section \ref{sec:Intrinsically defined patches and search regions}.

The integer lattice $\Z^k$ acts on $F_\rho$ by $n \cdot w \coloneqq \rho^*(n)+w$ for $n \in \Z^k$ and $w\in F_\rho$. For $r \in \R_+$ define the set of {\bf $r$-singular points} as
\[
\sing (r) = \W\cap (C(r) \cdot \partial \W).
\]
The {\bf $r$-regular points} are defined to be $\reg(r)=\W\setminus \sing(r)$. For $y \in Y$, we define $y^* \coloneqq \rho^*(\tilde{y})$. The map $y \mapsto y^*$ is sometimes called the {\bf star map}. It is instructive to observe that the lift of an $r$-patch $P(y,r)$ to the total space does not intersect the boundary of the strip $\mathcal{S}$ precisely when $y^*$ is $r$-regular. The following result relates the connected components of $\reg(r)$ to the collection of patches of size $r$, and will be an essential ingredient in the proofs of Theorems \ref{main:discrepancy}, \ref{main:transference} and \ref{thm:repulsive}. The lemma is formulated in \cite{HaynKoivWalt2015a}, and the proof can be found in \cite[Lemma 3.2]{HaynKoivSaduWalt2015}.
\begin{lemma}[Lemma 2.4 of \cite{HaynKoivWalt2015a}] \label{lem:connected components}
For a regular, cubical cut and project set, for every equivalence class $\PP$ of $r$-patches there is a unique connected component $Q$ of $\reg (r)$ such that, for any $y \in Y$,
\[
\PP(y, r) =\PP \textrm{ if and only if } y^* \in Q.
\]
\end{lemma}
We call $Q$ the {\bf acceptance domain} of the patch $\PP$. Through this lemma, and $E$ being totally irrational, an application of the Birkhoff Ergodic Theorem gives the following lemma.
\begin{lemma}[Lemma 3.2 of \cite{HaynKoivSaduWalt2015}]\label{lem:asymptotic}
The frequency $\xi_\PP$ of an equivalence class $\PP$ of $r$-patches is equal to $|Q|$, where $Q$ is the connected component $\reg(r)$ corresponding to $\PP$ of Lemma \ref{lem:connected components}.
\end{lemma}

\subsection{Discrepancy}

For a sequence $(x_n)_{n \in \N}$ of $\R^m$ and a measurable set $A \subseteq \R^m$, we define the {\bf discrepancy}
\[
D_N(A)=\left |\sum_{n=1}^N\chi_A(x_n) - N|A|\right |
\]
where $\chi_A$ stands for the characteristic function of $A$. The proof of Theorem \ref{main:discrepancy} hinges on estimates of this quantity. The following theorem is proved as \cite[Theorem 5.21]{Harm1998}, or in the current form \cite[p. 116]{KuipNied1974}.
\begin{lemma}[Erd\"os--Turan--Koksma inequality] \label{lem:A-T ineq}
Let $(x_n)_{n \in \N}$ be a sequence in $\R^m$. For any $L, N\in \N$,
\[
\sup_A \Big(\frac{D_N(A)}{N} \Big) \le C_m\left (\frac 1L + \sum_{\substack{0<\abs{h}\le L\\h\in\Z^m}}r(h)\left | \frac{1}{N} \sum_{n=1}^N\exp(2\pi i\inner{h} {x_n}) \right|\right),
\]
where the supremum is taken over all axes parallel boxes $A$,
\[
r(h)^{-1}=\prod_{j=1}^m \max\{1, \abs{h_j}\}
\]
for $h\in \R^m$, and $C_m$ is a constant only depending on $m$.

\end{lemma}

\subsection{Diophantine approximation}

Let $\psi: \N\to \R_+$ be a decreasing function. We say that a linear map $L:\R^m\to \R^n$ is {\bf $\psi$-badly approximable}, and write $L\in \B(\psi)$, if for some constant $C>0$
\[
\norm{L(q)} \ge C \psi(\abs{q}) \textrm{ for all } q \in \Z^m\setminus \{0\}.
\]
On the other hand, we say that it is {\bf $\psi$-well approximable}, and write $L\in \mathcal E(\psi)$, if $\norm{L(q)} \le \psi(\abs{q})$ for infinitely many $q\in \Z^m$. The Khintchine--Groshev theorem below connects the measure of $\mathcal{E}(\psi)$ to the speed of decay of $\psi$. The proof may be found in \cite [Section~12.1]{BerDicVel06}.

\begin{lemma}[Khintchine--Groshev]  \label{lem:Khintchine-Groshev}
The set $\mathcal E(\psi)$ has either full Lebesgue measure or measure $0$ according to whether the sum
\[
\sum_{r=1}^\infty r^{m-1}\psi(r)^{n}
\]
diverges or converges, respectively.
\end{lemma}
Notice that this gives the corresponding zero-one law for $\B(\psi)$ as well. The property of $L$ being $\psi$-badly approximable can be converted to well-distribution properties of the orbit of $L$ via the following transference principle:

\begin{lemma}[Theorem VI of Section V in \cite{Cass57}]\label{lem:transference}
Suppose that for some $\psi$ and $X>0$, there is no $n \in \Z^d\setminus\{0\}$ satisfying simultaneously
\[
\|L(n)\|\le \psi \textrm{ and } |n|\le X.
\]
Then for all $\gamma\in \R^{k-d}$, there is $n\in \Z^d$ with
\[
\|L(n) - \gamma\|\le c \textrm{ and }|n|\le R,
\]
where
\[
 c = \tfrac 12 (h+1)\psi, \quad R=\tfrac 12 (h+1)X, \quad \textrm{and}\quad h=\lfloor X^{-d}\psi^{d-k}\rfloor.
\]
\end{lemma}

\section{Quantitative estimates for frequencies of patterns}\label{sec:discrepancy proof}
In this section we prove Theorem \ref{main:discrepancy}, which bounds the rate of convergence of the estimates $\xi_\PP(y,R)$ to the asymptotic frequencies $\xi_\PP$. The proof will incorporate tools from discrepancy theory, reviewed in Section \ref{sec:toolbox}. We start with a technical lemma.
\begin{lemma}\label{lem:technical}
For almost every matrix $(\alpha_{ij}) \in \R^{d(k-d)}$, 
for any $\epsilon > 0$, we have that
\begin{equation}\label{eq:sum-lemma-to-bound}
\sum_{\substack{0<\abs{h}\le H\\h\in\Z^{k-d}}} r(h) \prod_{i=1}^d \| \inner {h} {(\alpha_{ij})_{j=1}^{k-d}} \|^{-1} \ll_{\epsilon, \alpha} (\log H)^{k+\epsilon},
\end{equation}
where $\inner h {(\alpha_{ij})_j}$ denotes the inner product of the vector $h \in \Z^{k-d}$ with the $i$-th row of $(\alpha_{ij})$, and
$r$ is defined as in the statement of Lemma~\ref{lem:A-T ineq},
\[
r(h)^{-1}=\prod_{j=1}^m \max\{1, \abs{h_j}\}.
\]
\end{lemma}

\begin{proof}
For $h \in (\R_+)^{k-d}$ define
\[
J(h) \coloneqq \int_{\R^{k-d}/\Z^{k-d}} \Bigl( \|\inner {h} {\beta} \| \cdot \bigl| \log \| \inner {h} {\beta} \| \bigr|^{1 + \delta} \Bigr)^{-1} d \beta.
\]
We claim first of all that $J(h) \ll 1$. Indeed, given $h = (h_j)_j$, let $i$ be fixed, such that $h_i \geq h_j$ for any $j \neq i$.
Consider the change of basis defined by $u_i \coloneqq \inner{h}{\beta}$ and $u_j \coloneqq \beta_j$ for $i \neq j$. This change of basis has Jacobian determinant $h_i^{-1}$, and the domain $[0,1)^{k-d}$ transforms to a region contained in the box $B = [0,l_1) \times \cdots \times [0,l_{k-d})$, where $l_j = 1$ for $i \neq j$ and $l_i = (k-d) h_i$. It follows that
\[\begin{split}
J(h) & \leq \int_B \bigl( h_i \cdot \|u_i\| \cdot |\log \|u_i\||^{1 + \delta} \bigr)^{-1} d u \\
     & =    (k-d) \biggl(\int_0^{{1}/{2}} (-u_i \cdot (\log u_i)^{1+\delta})^{-1} du_i \\
     & \qquad      + \int_{{1}/{2}}^1 (-(1-u_i) \cdot (\log (1-u_i))^{1+\delta})^{-1} du_i \biggr) \\
     & =    2 (k-d) \delta^{-1} (\log 2)^{-\delta}.
\end{split}\]

From $J(h) \ll 1$ we may deduce that
\[
\sum_{h_1 = 1}^\infty \cdots \sum_{h_{k-d} = 1}^\infty (h_1 (\log h_1)^{1+\delta} \cdots h_{k-d} (\log h_{k-d})^{1+\delta})^{-1} \prod_{j=1}^d J(h) < \infty.
\]
Exchanging the order of summation and integration, the quantity
\begin{multline*} \label{eq:sum}
 X \coloneqq \sum_{h_1 = 1}^\infty \cdots \sum_{h_{k-d} = 1}^\infty (h_1 (\log h_1)^{1+\delta} \cdots h_{k-d} (\log h_{k-d})^{1+\delta})^{-1} \cdot \\
\prod_{j=1}^d (\|\inner {h} {\alpha} \| \cdot |\log \| \inner {h} {\alpha}\|)^{1+\delta}|^{-1}
\end{multline*}
is bounded for almost every choice of $\alpha =(\alpha_j) \in \mathbb{R}^{k-d}$. For such an $\alpha$, we then have that
\[\begin{split}
A'(H) & \coloneqq \sum_{h_1 = 1}^H \cdots \sum_{h_{k-d} = 1}^H (h_1 \cdots h_{k-d})^{-1} \prod_{j=1}^d \|\inner{h}{\alpha}\|^{-1}  \\
      & \leq X \cdot \max_{0 < |h| \leq H} \Big((\log h_1 \cdot \cdots \log h_{k-d})^{1+\delta} \prod_{j=1}^d | \log \| \inner{h}{\alpha}\||^{1+\delta} \Big) \\
      & \ll_{\delta, \alpha} (\log H)^{(k-d)(1+\delta)} \prod_{j=1}^d \max_{0 < |h| \leq H} \bigl| \log \| \inner{h}{\alpha}\| \bigr|^{1+\delta}.
\end{split}\]
By the Khintchine--Groshev theorem (Lemma~\ref{lem:Khintchine-Groshev}) applied to one linear form in $k-d$ variables, there is a full measure set of $\alpha \in \mathbb{R}^{k-d}$ for which, for all $\epsilon > 0$, there exists $C>0$ with
\[
\|\inner {h} {\alpha_j}\| \geq \frac{C}{|h|^{k-d+\epsilon}}
\]
for all non-zero $h \in \mathbb{Z}^{k-d}$. It follows that for almost all $\alpha$ we may bound
\[
A'(H)\ll (\log H)^{(k-d)(1+\delta) + d(1+\delta)} \ll (\log H)^{k+\epsilon}
\]
for any $\epsilon > 0$ by setting $\delta$ sufficiently small. The implicit constant depends on $\epsilon$ and $\alpha$.

The above calculation may be repeated for the sums analogous to $A'(H)$ but where certain indices run over negative values.
For non-empty $S \subseteq \{1,\ldots,d\}$, let $A_S$ be the sum given by Equation~\eqref{eq:sum-lemma-to-bound}, but where we only sum over vectors $h$ which are non-zero in those coordinates belonging to $S$.
By the above, the quantity $A_S$ is bounded by $2^{\# S} \cdot A'(H) \ll (\log H)^{k+\epsilon}$. Since the sum which we want to bound is equal to $\sum_{S \in 2^{\{1,\ldots,d\}}} A_S \ll (\log H)^{k+\epsilon}$, the lemma follows.
\end{proof}

We are now ready to prove Theorem \ref{main:discrepancy}:

\begin{proof}[Proof of Theorem \ref{main:discrepancy}]
Let $L : \R^d \to \R^{k-d}$ and a corresponding regular, cubical cut and project set $Y = Y_s$ be given. For any equivalence class of $r$-patch $\PP$ and $y \in Y$, by Lemma \ref{lem:connected components} we have that $\PP(y,r) = \PP$ if and only if $y^* \in Q$, where $Q$ is the connected component of $\reg(r)$ corresponding to $\PP$, which is an axes parallel box. By Lemma \ref{lem:asymptotic} we have that $\xi_\PP = |Q|$. For $N \in \N$, let
\[
\chi_\PP(y,N) \coloneqq \{y' \in Y \mid \PP(y',r) = \PP \text{ and } \tilde{y}' - \tilde{y} \in C(N)\},
\]
so that we wish to bound the quantity \[D_N(\PP) \coloneqq |\#\chi_\PP(y,N) - (2N+1)^d \cdot \xi_\PP|.\]

Since the cubical window $\W$ with some boundary points removed is a fundamental domain for $\{0\}^d \times \Z^{k-d}$ in $F_\rho$, we may identify $\chi_\PP(y,N)$ with the set
\begin{equation}\label{eq:return}
\{n \in [-N,N]^d_{\Z} \mid \LL(n)+y^* \in Q\}.
\end{equation}
By Lemma \ref{lem:A-T ineq}, there is a uniform constant $C>0$, independent of $\PP$, for which
\[
\frac{D_N(\PP)}{(2N)^d} \leq C \Big(\frac{1}{H} + \sum_{0<\abs{h}\le H} \frac{r(h)}{(2N+1)^d} |S| \Big)
\]
for any $H \in \N$, where
\begin{align*}
 S         & = \sum_{\substack{n \in \Z^d \\ |n| \leq N}} \exp(2\pi i\inner{h} {\LL(n)}),
\end{align*}
and
\begin{align*}
 r(h)^{-1} & = \prod_{j=1}^{k-d} \max \{1,|h_j|\}.
\end{align*}
An upper bound for the exponential sum may be given as
\[\begin{split}
|S| & = \left| \sum_{n_1 = -N}^N \cdots \sum_{n_d = -N}^N \exp(2\pi i\inner{h} {\LL(n_1, \ldots, n_d)}) \right| \\
    & \leq \prod_{i=1}^d \frac{2}{|1- \exp(2\pi i\inner{h} {\LL(e_i)})|} \\
    & = \prod_{i=1}^d \frac{2}{2 \left| \sin( \pi \inner{h} {\LL(e_i)}) \right|}
     \leq \prod_{i=1}^d \bigl(2 \| \inner {h} {\LL(e_i)} \| \bigr)^{-1},
\end{split}\]
using the concavity inequality $\left| \sin(\pi x) \right| \geq 2 \left| x \right|$ on $[-1/2, 1/2]$.
This reveals how the discrepancy may be controlled by restricting the Diophantine properties of $L$. By Lemma \ref{lem:technical}
\[
\sum_{0<\abs{h}\le H} r(h) \prod_{i=1}^d \| \inner {h} {\LL(e_i)} \|^{-1} \ll_{\delta, L} (\log H)^{k+\delta}
\]
for any $\delta > 0$ for almost every $L$. Hence
\[
D_N(\PP) \ll_{\delta,L} \frac{N^d}{H} + (\log H)^{k+\delta}.
\]
Letting $H = N^d$, we have that
\[
D_N(\PP) \ll_{\epsilon,L} (\log N)^{k + \epsilon}
\]
for any $\epsilon > 0$. It easily follows that there exists some $C>0$ for which $D_R(\PP) < C(\log R)^{k+\epsilon}$ for any $R \geq 1$.
\end{proof}

\begin{remark}\label{rem:axes parallel}
Notice that in order to use Lemma \ref{lem:A-T ineq} in the proof of Theorem \ref{main:discrepancy} it is only necessary to know that the appearance of an $r$-patch $\PP$ corresponds to a visit under $\mathcal L$ to some axes parallel box, as in \eqref{eq:return}. This fact will be needed in Section \ref{sec:Intrinsically defined patches and search regions}.
\end{remark}

The proof of low discrepancy established in the above argument may be used to bound the repetitivity function for typical cut and project sets. We deduce the following corollary to Theorem \ref{main:discrepancy}, which gives a slight weakening on the first bound of the repetitivity function given in Theorem \ref{main:transference}:

\begin{corollary} Fix $\epsilon>0$. For Lebesgue almost all linear maps $L:\R^d\to \R^{k-d}$, the corresponding cubical cut and project sets are $\phi$-repetitive for
\[
\phi(r) \ge Cr^{k-d}(\log r)^{\frac{2k}{d} - 1 + \epsilon}.
\]
The constant $C$ depends on $L, \epsilon$ and $\pi$.
\end{corollary}

\begin{proof}[Proof of Theorem \ref{main:transference}]
To obtain a lower bound for typical repetitivity, we wish to firstly bound the sizes of the regions $Q$ of Lemma \ref{lem:connected components} from below. This will dictate the long-term behaviour of appearances of patches across the resulting cut and project sets. To this end, let $\psi(n) \coloneqq (n^d(\log n)^{1+\epsilon_1})^{-1}$ and consider $L=(L_1, \dots, L_{k-d})$, written as $(k-d)$ linear forms in $d$ variables, with each $L_j\in \B(\psi)$.
The set of such $L$ is full measure by the Khintchine--Groshev Theorem \ref{lem:Khintchine-Groshev}.
For $n_1,n_2 \in [-N,N]^d_\Z$, since each $L_i \in \B(\psi)$, we have that
$|\LL_j(n_1) - \LL_j(n_2)|\ge \norm{\LL_j(n_1 - n_2)}\ge C\psi(r)$, where $C$ only depends on $L$.
It follows that the volumes of the connected components of $\reg(r)$ are bounded from below by a function that grows at least as fast as $(r^{d}(\log r)^{1+\epsilon_1})^{-(k-d)}$.

By the discrepancy estimate of Theorem~\ref{main:discrepancy}, there exist constants $c_1, c_2$, which do not depend on $\PP$, $y$ or $R$, for which $\xi_\PP(y,R)$ satisfies the estimate
\begin{equation}\label{eq:freq-minus-discrep}
\xi_\PP - c_2 \frac{(\log R)^{k+\epsilon_2}}{R^d} \geq c_1 (r^{d}(\log r)^{1+\epsilon_1})^{-(k-d)} - c_2 \frac{(\log R)^{k+\epsilon_2}}{R^d},
\end{equation}
where we use the fact that the frequency $\xi_\PP$ is given by the volume of a connected component of $\reg(r)$.
Recall that $\xi_\PP(y,R)$ counts the number of occurrences of $\PP$ in a region of size $R$ about $y$. If we pick $R = r^{k-d}(\log r)^{(2k/d) - 1 + \epsilon}$, where $\epsilon$ is some positive number that can be made arbitrarily small by setting $\epsilon_1$, $\epsilon_2$ sufficiently small, we deduce that the quantity of \eqref{eq:freq-minus-discrep} is eventually strictly positive as $r$ grows.
\end{proof}

\section{Repulsivity and repetitivity}\label{sec:repetitivity}

\begin{proof}[Proof of Theorem \ref{main:transference}] As in the previous proof, we will first want to investigate the sizes of the connected components of $\reg(r)$. Let $L=(L_1,\ldots,L_{k-d})$ be a system of $k-d$ linear forms $L_i$ of $d$ variables, where each $L_i \in \B(\psi)$ with $\psi(r) \coloneqq (r^d (\log r)^{1+\epsilon})^{-1}$. By Lemma \ref{lem:Khintchine-Groshev}, a set of full measure of $L$'s satisfies this condition. By the Diophantine conditions on each $L_i$ and Lemma \ref{lem:connected components}, the connected components $Q$ corresponding to $r$-patches of $Y$ are boxes whose side lengths are bounded below by a constant times $(r^d (\log r)^{1+\epsilon})^{-1}$.

Limiting the long-term frequency of appearances of a patch does not preclude it appearing multiple times in a smaller region than expected, and then not appearing at all in larger regions. To curtail this sort of behaviour, in addition to the conditions above, we also wish to enforce well-distribution of $\LL$ in $[0,1)^{k-d}$. So we suppose that $L \in \B(\overline\psi)$ where $\overline\psi(\phi) \coloneqq (\phi^{\frac{d}{k-d}} (\log \phi)^{\frac{1+\epsilon'}{k-d}})^{-1}$ for any $\epsilon' > 0$. By Lemma \ref{lem:Khintchine-Groshev}, this property (in conjunction with the property above on each $L_i$) applies to a full measure set of linear forms. By transference, $\LL$ applied to a box of integers of side length $\phi$ has density a constant times
\[
\phi^{-\frac{d}{k-d}} (\log \phi)^{\frac{d-1}{k-d} + 1 + \delta}
\]
in $[0,1)^{k-d}$, where $\delta > 0$ can by made arbitrarily small by setting $\epsilon'$ sufficiently small.

By Lemma \ref{lem:connected components}, there is a constant $C$ depending only on $E$ for which, whenever each $\phi$-orbit $x + \LL([-\phi,\phi]_{\Z}^d)$ of $x \in [0,1)^{k-d}$ intersects each connected component $Q$ of $\reg(r)$, then we have that every $r$-patch of $Y$ occurs within distance $\phi$ of every point $y \in Y$. By the calculations above, given $r>0$, it is sufficient to set $\phi = \phi(r)$ so that
\[
\phi^{-\frac{d}{k-d}} (\log \phi)^{\frac{d-1}{k-d} + 1 + \delta} \leq C (r^d (\log r)^{1+\epsilon})^{-1},
\]
for some constant $C>0$ (depending only on $E$ and $\epsilon$). A quick calculation shows that we may choose
\[
\phi(r) = c r^{k-d} (\log r)^{\frac{2k - 1}{d} -1 + \delta'}
\]
for constants $c,\delta' > 0$ which only depend on $E$ and $\epsilon$, and for which $\delta'$ can be made arbitrarily small by setting $\epsilon$ sufficiently small.

For the other bound on the typical behaviour of the repetitivity function, let $\psi(r) = c_1 r^{-d} (\log r)^{-1}$ and $L$ be such that $L_1 \in \E(\psi)$ and each of the $L_i$ have trivial kernels. The set of such linear forms is full measure by Lemma \ref{lem:Khintchine-Groshev}. It follows that for infinitely many values of $r$, there exists some acceptance domain $Q$ for an $r$-patch $\PP$ which is a box with first side length less than $c_1 r^{-d} (\log r)^{-1}$ and other sides, by a simple counting argument, of length less than $r^{-d}$. It will follow that any cubical cut and project set associated to $L$ is not $\phi$-repetitive so long as $\phi$ is chosen so that the orbit of $\LL$ applied to a box of integers of size $\phi(r)$ has gaps larger than these boxes.

Given $h,v > 0$, let $\phi = (h \cdot v^{k-d-1})^{1/d}$. Then for any positive $\alpha < 1$, for sufficiently large $\phi$ we may subdivide $[0,1)^{k-d}$ into more than $\alpha \phi^d$ boxes whose first side lengths are bounded below by $h$, and others are bounded below by $v$. So there exists $c_2$ for which the orbit under $\LL$ over a box of integers of size $c_2 \phi$ must fail to visit some box whose first side is $v$ and others are $h$. Set $v = c_1 r^{-d} (\log r)^{-1}$ and $h = r^{-d}$, so that $\phi = \phi(r) = c_1^{-1/d} r^{k-d} (\log r)^{1/d}$. Choosing appropriate starting points in the cut and project set (the positions of which correspond to a dense subset of $\W$, by irrationality), we may arrange for the orbit of $\LL$ under a box of integers of size $c_2 \phi(r)$ to miss the acceptance domain of $r$-patches $\PP$, for infinitely many values of $r$. Since the constant $c_1$ was arbitrary and $c_2$ is fixed, it follows that almost every cut and project set is not $\phi$-repetitive with $\phi(r) = C r^{k-d} (\log r)^{-1/d}$, for any $C>0$. \end{proof}

\begin{remark} The power of the logarithm for the lower bound in Theorem \ref{main:transference} is likely to be far from optimal for $k-d > 1$, since the argument only exploits Diophantine properties of the linear forms in a single direction, implementing trivial bounds in the others. \end{remark}

We now turn our attention to repulsivity.
\begin{proof}[Proof of Theorem \ref{thm:repulsive}]
Almost every linear form $L \colon \R^d \rightarrow \R^{k-d}$ is $\psi$-well approximable with
\[
\psi(r) = cr^{-\frac{d}{k-d}}(\log r)^{-\frac{1}{k-d}},
\]
for any $c>0$, and almost every linear form $L$ also has the property that each $L_i$ has trivial kernel. Given such an $L$, let non-zero $n \in \Z^d$ satisfy $\|L(n)\| \leq \psi(|n|)$ and set $r \coloneqq 2^{-\frac{1}{d}}\psi(|n|)^{-\frac{1}{d}}$. By a simple counting argument, for any positive $\alpha < 1$, for sufficiently large $r$ there exists a connected component $Q$ of $\reg(r)$ which is a box with side lengths bounded below by $r^{-d} = \alpha \psi(|n|)$.

It follows that for $m \in \Z^k + s$ projecting sufficiently close to the centre of $Q$ (which exists by total irrationality), we have $\PP(\pi(m),r) = \PP(\pi(m+n+f),r)$ for some $f \in F_\rho \cap \Z^k$. This gives us the bound \[d(\pi(m),\pi(m+n+f)) \ll |n| \ll r^{k-d} \log(r)^{-1 / d}.\] It follows that a cubical cut and project set corresponding to $L$ is not $\phi$-repulsive with $\phi(r) \coloneqq cr^{k-d} \log(r)^{-1/d}$, for any $c > 0$.

For the other bound on the repulsivity function, we want to use typical Diophantine properties for $L$ to bound, firstly, the sizes of the connected components of $\reg(r)$ from above and, secondly, $\|L(n)\|$ from below. For the latter we may impose that $L$ is $\psi$-badly approximable with $\psi(r) = cr^{-\frac{d}{k-d}} \log(r)^{-\frac{1 + \epsilon}{k-d}}$ for all $c,\epsilon > 0$, a condition which is satisfied by almost all linear maps $L$. For the former, we may impose that each $L_i$ is $\psi$-badly approximable for $\psi(r) = cr^{-d} \log(r)^{-(1 + \epsilon)}$ for all $c,\epsilon > 0$, which again applies to almost all linear maps $L_i : \R^d \rightarrow \R$. By the transference principle of Lemma \ref{lem:transference}, for all $\epsilon > 0$ the connected components of $\reg(r)$ have side lengths bounded above by $cr^{-d} \log(r)^{d(1 + \epsilon)}$ for some $c = c_\epsilon > 0$.

We wish to show that, for sufficiently large $r$, whenever $d(y,y') < cr^{k-d}\log(r)^{-1/d - (k-d) -\epsilon}$ for distinct $y,y' \in Y$, then $\PP(y,r) \neq \PP(y',r)$. So let $m \neq n \in \Z^d+s$ with $|m-n|\le C'r^{k-d}\log(r)^{-1/d - (k-d) -\epsilon}$ for some constant $C'$ and $\LL(m+s)$ and $\LL(n+s)$ both belonging to the same connected component $Q$ of $\reg(r)$. By the above bounds on the side lengths of $Q$, we see that $\|L(m-n)\| \leq cr^{-d} \log(r)^{d(1+\epsilon)}$. We may now use our badly approximable hypothesis on $L$ to conclude that $|m-n|$ must be larger than some constant times $r^{k-d} \log(r)^{-1/d - (k-d) -\epsilon}$.
\end{proof}

\section{Intrinsically defined patches and search regions}\label{sec:Intrinsically defined patches and search regions}

\subsection{Other patch types}
In the above proofs it was advantageous to use a specific choice of notion of a patch. However, versions of Theorems \ref{main:discrepancy}, \ref{main:transference} and \ref{thm:repulsive} hold true for many other choices as well.

Notice that a choice of notion of $r$-patch is essentially a choice of equivalence relation $\simeq_r$ on the points of $Y$ for each $r \in \R_+$. Say that two such choices $\simeq^1_*$ and $\simeq^2_*$ are {\bf linearly equivalent} if there exist constants $A,c>0$ for which $\simeq^1_{Ar+c} \subseteq \simeq^2_r$ and $\simeq^2_{Ar+c} \subseteq \simeq^1_r$ (where an equivalence relation on $Y$, here, is considered as a certain subset of $Y \times Y$). For $\simeq^1_*$ and $\simeq^2_*$ linearly equivalent, it is easy to see that if $Y$ is $\phi$-repetitive with respect to $\simeq^1_*$, then it is $\phi'$-repetitive with respect to $\simeq^2$, for $\phi'(r) = \phi(Ar+c)$, and similarly in the other direction. A similar statement holds for $\phi$-repulsivity. So Theorems \ref{main:transference} and \ref{thm:repulsive} hold for any notion of $r$-patch linearly equivalent to the one introduced in Section \ref{sec:toolbox}.

In this section, we shall focus on the following two natural definitions for a patch of size $r$ at $y\in Y$:
\begin{align*}
P_\I (y, r\Omega) \coloneqq & \{y'\in Y\mid y'-y\in r\Omega\}; \\
P_\II (y, r\Omega) \coloneqq & \{y'\in Y\mid \tilde y'-\tilde y\in \rho^{-1}(r\Omega)\}.
\end{align*}
We call these types of patches {\bf type I} and {\bf type II patches}, respectively. Here $\Omega$ is some bounded convex subset of $E$ containing a neighbourhood of the origin. For $X \subseteq E$, let $N_\kappa(X)$ denote the {\bf $\kappa$-neighbourhood of $X$} of points of $E$ within $\kappa$ of $X \subseteq E$. Assuming that $\Omega$ is convex (amongst many other weaker conditions) we have the bound $|N_\kappa(\partial r \Omega)| \leq cr^{d-1}$, for sufficiently large $r$. This is required in the proof of Lemma \ref{lem: controlled complements} below.

The patches of Section \ref{sec:toolbox} are given by
\[
P(y, r)=P_\II(y, (r[-1, 1]^d + F_\rho) \cap E)).
\]
Hence patches of type II and patches from Section \ref{sec:toolbox} are linearly equivalent, and Theorems \ref{main:transference} and \ref{thm:repulsive} apply to them. It was noted in \cite{HaynKoivSaduWalt2015} as Equation 4.1 that there is a constant $c>0$ such that for any $y \in Y$ and $r>0$ large enough,
\[
P_\I(y,(r-c)\Omega)\subseteq P_\II(y,r\Omega)\subseteq P_\I(y, (r+c)\Omega).
\]
By this observation Theorems \ref{main:transference} and \ref{thm:repulsive} also apply directly to the intrinsically defined patches of type I.

In the following, we sometimes use the subindices to distinguish between patch types, and $\both$ when the statement holds for both type I and type II. The objects $\PP$, $\xi_\PP$ and $\xi_\PP(y,R)$ for type I and type II patches are defined as in Section \ref{sec:toolbox}, and for the most part the same notation is used, which should not be a cause of confusion. For example, two $r$-patches $P_\both(y_1,r\Omega)$ and $P_\both(y_2,r\Omega)$ of either type I or type II are equivalent if $P_\both(y_1,r\Omega) - y_1 = P_\both(y_2,r\Omega) - y_2$, and we denote the corresponding equivalence class by $\PP_\both(y_1,r\Omega)$.

The rest of this subsection is devoted to proving the following version of Theorem \ref{main:discrepancy} for patches of types I and II. It may be paraphrased as saying that the same discrepancy estimates hold for generalised patch types, but that for type I patches of size $r$ the constant term depends on $r$.
\begin{theorem}\label{thm:other_patches}
Let $\epsilon> 0$. Then for almost all choices of linear maps $L:\R^d\to \R^{k-d}$, for the corresponding cubical cut and project sets $Y$, there is a constant $C$ that only depends on $L, \epsilon$ and $\pi$ such that the following holds: Fix a bounded convex set $\Omega \subseteq E$ containing a neighbourhood of the origin. Let $r>0$, and let $y'\in Y$. Then for type II patches $\PP_\II = \PP_\II(y', r\Omega)$, for all $y\in Y$, and for all $R \geq 1$
 \[
  \abs{\xi_{\PP_\II} (y,R) - \xi_{\PP_\II}} \leq C \cdot \frac{\log(R)^{k+\epsilon}}{R^d}.
 \]
Furthermore, for $r$ large enough, for type I patches $\PP_\I = \PP_\I(y', r\Omega)$, for all $y \in Y$, and for all $R \geq 1$
 \[
  \abs{\xi_{\PP_\I} (y,R) - \xi_{\PP_\I}} \leq C \cdot \frac{\log(R)^{k+\epsilon} r^{(d-1)(k-d-1)}}{R^d}.
 \]
\end{theorem}

The proof of this claim will follow from the proof of Theorem~\ref{main:discrepancy}, along with some control on what we call {\bf acceptance domains} associated to type I and II patches, see Lemma \ref{lem: controlled complements}. Set $\widetilde{\calS}=(\mathcal{W}-\mathcal{W})+E$ and, given a fixed patch shape $\Omega$, let $Z_\I(r\Omega) \coloneqq \Z^k \cap \widetilde{\calS} \cap \pi^{-1}(r\Omega)$ and $Z_\II(r\Omega) \coloneqq \Z^k \cap \widetilde{\calS} \cap \rho^{-1}(r\Omega)$. Recall the definition of the star map from Section \ref{sec:toolbox}.

\begin{lemma}\label{lem: acceptance domains}
Let $P_\both(y,r\Omega)$ be a patch of type I or II. We have that $y' \in P_\both(y,r\Omega)$ if and only if $y' = y + \pi(n)$ for $n \in Z_\both(r\Omega)$ satisfying $y^* \in \W - n^*$.
\end{lemma}

\begin{proof} Suppose that $y' \in P_\both(y,r\Omega)$. Then $y,y' \in Y$ or, equivalently, $\tilde{y}, \tilde{y}' \in \calS \cap (\Z^k+s)$. It follows that $n \coloneqq \tilde{y}'-\tilde{y} \in \widetilde{\calS} \cap \Z^k$. If $y' \in P_\text{I}(y,r\Omega)$ then $y'-y = \pi(n) \in r\Omega$, so $n \in Z_\I(r\Omega)$. Similarly, if $y' \in P_\II(y,r\Omega)$ then $\tilde{y}'-\tilde{y} = n \in \rho^{-1}(r\Omega)$, so $n \in Z_\II(r\Omega)$. It follows that for $y' \in P_\both(y,r\Omega)$ it is necessary that $y'-y=\pi(n)$ with $n \in Z_\both(r\Omega)$. Assuming that there is such an $n$, we have that $y' \in P_\both(y,r\Omega)$ if and only if $y'=y+\pi(n) \in Y$; equivalently, $\tilde{y}+n \in \calS$ which is the case if and only if $y^* \in \W-n^*$.
\end{proof}

The above allows us to construct acceptance domains for patches. Given an equivalence class of patch $\PP = \PP_\both(y,r\Omega)$ of type I or II, define
\begin{align*}
Z_\both^\in(\PP) \coloneqq & \{n \in Z_\both(r\Omega) \mid n = \tilde{y}' - \tilde{y} \text{ for some } y' \in P_\both(y,r\Omega)\},
\end{align*}
and
\begin{align*}
Z_\both^{\notin}(\PP) \coloneqq & \{n \in Z_\both(r\Omega) \mid n \neq \tilde{y}' - \tilde{y} \text{ for any } y' \in P_\both(y,r\Omega)\}.
\end{align*}
Of course, these sets do not depend on the choice of representative $P_\both(y,r\Omega)$ of $\PP$. Notice that $Z_\both^\in(\PP)$ and $Z_\both^{\notin}(\PP)$ are complementary subsets of $Z_\both(r\Omega)$. To explain the logic of the notation, notice that the elements of $Z_\both^\in(\PP)$ determine which lifted points are \emph{in} $\PP$, relative to the central point of the patch, and $Z_\both^{\notin}(\PP)$ determines which, of the points which \emph{could} be in $\PP$, are in fact \emph{not} in $\PP$.

For a subset $X$ of (an understood) space $\Xi$, we let $X^\text{c}$ be the closure of the complement of $X$ in $\Xi$, that is, $X^\text{c} \coloneqq \overline{\Xi \setminus X}$.

\begin{corollary} Let $\PP = \PP_\both(y_1,r\Omega)$ be a patch of type~I or~II. Then there exists a subset $A(\PP)$ of the window for which, for any $y_2 \in Y$, we have that $\PP_\both(y_2,r\Omega) = \PP$ if and only if $y_2^* \in A(\PP)$. Moreover, we may set
\begin{equation} \label{eq:acceptance-domain}
A(\PP) = \bigcap_{n \in Z_\both^\in(\PP)} \W - n^* \cap \bigcap_{n \in Z_\both^{\notin}(\PP)} \W^\text{c} - n^*.
\end{equation}
\end{corollary}

\begin{proof} Suppose that $y_2 \in Y$ with $y_2^* \in A(\PP)$. So for all $n \in Z_\both^\in(\PP)$ we have that $y_2^* \in \W - n^*$, and hence $y_2 + \pi(n) \in P_\both(y_2,r\Omega)$ by Lemma \ref{lem: acceptance domains}. Again by Lemma \ref{lem: acceptance domains}, for all $y' \in P_\both(y_1,r\Omega)$ we have that $y' = y_1 + \pi(n)$ for some $n \in Z_\both^\in(\PP)$. It follows that $P_\both(y_1,r\Omega) - y_1 \subseteq P_\both(y_2,r\Omega) - y_2$. The opposite inclusion is similar, using $Z_\both^{\notin}(\PP)$ in place of $Z_\both^\in(\PP)$. Conversely, suppose that $y_2^* \notin A(\PP)$. It follows that $y_2^* \notin \W - n^*$ for some $n \in Z_\both^\in(\PP)$ or $y_2^* \notin \W^\text{c} - n^*$ for some $n \in Z_\both^{\notin}(\PP)$. By Lemma \ref{lem: acceptance domains}, in the former case we have that $y_2 + \pi(n) \notin P_\both(y_2,r\Omega)$ but that $y_1 + \pi(n) \in P_\both(y_1,r\Omega)$, and in the latter case we have that $y_2 + \pi(n) \in P_\both(y_2,r\Omega)$ but that $y_1 + \pi(n) \notin P_\both(y_1,r\Omega)$. \end{proof}

We call the region $A(\PP)$ constructed above the {\bf acceptance domain} of $\PP$. This is an extension of the definition in Lemma~\ref{lem:connected components} to more general patches.
It is again a consequence of the Birkhoff Ergodic Theorem and the irrationality of $E$ that $\xi_\PP = |A(\PP)|$. In order to prove discrepancy estimates for these regions, analogously to our proof of Theorem \ref{main:discrepancy}, we need more control over the shapes of these regions. This amounts to reducing the number of translates of the complement of the window required on the right-hand side of the intersection of (\ref{eq:acceptance-domain}). The following lemma shows that we may completely eliminate these entries for type II patches, so that the corresponding acceptance domains are axes parallel boxes, and for type I patches we only require a number of translates which grows at the rate of the measure of a neighbourhood of the boundary of $r\Omega$.

\begin{lemma}\label{lem: controlled complements}
For a patch $\PP = \PP_\I(y,r\Omega)$ of type I, we have that
\[
A(\PP) = \bigcap_{n \in Z_\both^\in(\PP)} \W - n^* \cap \bigcap_{n \in Z'} \W^\text{c} - n^*
\]
where $\# Z' \leq cr^{d-1}$ for some constant $c$ depending only on $\Omega$, $L$ and $\pi$. For a patch $\PP = \PP_\II(y,r\Omega)$ of type II, we have that
\[
A(\PP) = \bigcap_{n \in Z_\both^\in(\PP)} \W - n^*
\]
and so $A(\PP)$ is an axes parallel box.
\end{lemma}

\begin{proof} Since $\W$ is a fundamental domain for $\{0\}^d \times \Z^{k-d}$ in $F_\rho$, we may express the closure of the complement of the window in $\W + (\W - \W)$ via the identity
\[
\W^\text{c} \cap (\W + (\W - \W)) = \big( \bigcup_{m \in K} \W + m^* \big) \cap (\W + (\W - \W)),
\]
where $K \subset \{0\} \times \Z^{k-d}$ is a finite set (in particular, it is the set of $3^{k-d} - 1$ sums of the form $\sum_{i=1}^{k-d} \epsilon_i \cdot e_{i+d}$ where the $\epsilon_i \in \{-1,0,+1\}$, not all $0$, and the $e_{i+d}$ are the $k-d$ standard basis vectors of $F_\rho$). Since $n^* \in \W-\W$ for all $n \in Z_\both^{\notin}(\PP)$ and $A(\PP) \subseteq \W$, we may replace the occurrences of $\W^\text{c}$ in~\eqref{eq:acceptance-domain} by the union above, giving
\[
A(\PP) = \bigcap_{n \in Z_\both^\in(\PP)} \W - n^* \cap \bigcap_{n \in Z_\both^{\notin}(\PP)} \big(\bigcup_{m \in K} \W - (n-m)^*\big).
\]
So, for all $n \in Z_\both^{\notin}(\PP)$, there exists some $m \in K$ for which $A(\PP)$ has non-trivial intersection with $\W-(n-m)^*$. If $(n-m) \in Z_\both^\in(\PP)$, then clearly removing $n$ from $Z_\both^{\notin}(\PP)$ does not change the intersection. By \eqref{eq:acceptance-domain}, since $A(\PP) \subseteq W-(n-m)^*$, we may remove $n$ from the list if $(n-m) \in Z_\both(r\Omega)$. Since $\W$ intersects $\W - (n-m)^*$, we have that $(n-m)^* \in \W-\W$, so it suffices to check that $(n-m) \in \pi^{-1}(r\Omega)$ in the case of a patch of type I and that $(n-m) \in \rho^{-1}(r\Omega)$ for a patch of type II.

In the latter case we have that $n \in \rho^{-1}(r\Omega)$ and, since $m \in F_\rho$, we still have that $(n-m) \in \rho^{-1}(r\Omega)$. Therefore all elements of $Z_\II^{\notin}(\PP)$ may be removed in (\ref{eq:acceptance-domain}) without changing the intersection. In short, the existence of certain points not being in $\PP_\II$ is automatically ensured by corresponding points being in $\PP_\II$.

For patches of type I, if $(n-m) \notin \pi^{-1}(r\Omega)$ then $n \in \pi^{-1}(r\Omega)^\text{c} + K$. Thus we wish to bound the size of the set
\[
Z_\I(r\Omega) \cap (\pi^{-1}(r\Omega)^\text{c} + K) = \Z^k \cap \widetilde{S} \cap \pi^{-1}(r\Omega \cap (r\Omega^\text{c} + \pi(K))) = \Z^k \cap X_r
\]
where $X_r = \widetilde{S} \cap \pi^{-1}(r\Omega \cap (r\Omega^\text{c} + \pi(K)))$. Consider the region $X'_r$ given as the union of unit boxes centred at the points of $\Z^k$ intersecting $X_r$ non-trivially. The number of lattice points in $X_r$ is bounded by the measure of $X'_r$. Notice that
\begin{align*}
\rho^*(X_r) & \subseteq \W-\W,~\text{and} \\
\pi(X_r) & \subseteq r\Omega \cap (r\Omega^\text{c} + \pi(K)) \subseteq N_{\kappa_1}(\partial r \Omega).
\end{align*}
Since $\rho^*$ and $\pi$ are complementary, and $X'_r \subseteq N_{\kappa_2}(X_r)$ (where $\kappa_2$ is simply the length of the diagonal of the unit cube in $\R^k$), it follows that
\[
\# (\Z^k \cap X_r) \leq |X'_r| \leq C |N_{\kappa_3}(\partial r \Omega)|,
\]
where $C$ and $\kappa_3$ are constants which depend only on $\Omega$, $L$ and $\pi$. It is not difficult to show that for a convex set $\Omega$, given $\kappa > 0$, there exists some $c$ for which $|N_\kappa(\partial r\Omega)| \leq c r^{d-1}$ for sufficiently large $r$. Since we need only include those elements of $\Z^k \cap X_r$ in $Z_\I^{\notin}(\PP)$ in the intersection defining $A(\PP)$, and since $\#(\Z^k \cap X_r) \leq cC r^{d-1}$, the result follows. \end{proof}

\begin{lemma}
Let $A_0 \subseteq \R^m$ be an axes parallel box with side lengths at most $1$, and let $A$ be a region obtained by removing $N$ translates of the unit cube from $A_0$, that is,
\[
A=A_0\setminus\bigcup_{i=1}^N([0,1)^m+x_i),
\]
where $x_i\in \R^m$. Then $A$ can be written as a union of $(N+1)^{m-1}$ axes parallel boxes, only overlapping on their boundaries.
\end{lemma}
\begin{proof}
The claim may be proved inductively over the dimension $m$. The statement is obvious when $m=1$, since then $A$ is an interval. So suppose for the inductive step that the claim holds in dimension $m$, and $A \subseteq \R^{m+1}$. For each of the $N$ translates of the cube there is (at most) one face that is orthogonal to the $(n+1)$st basis vector and has a non-empty intersection with $A$. Denote the corresponding hyperplanes containing these faces by $H_1,\ldots,H_K$. Let $H_0$ and $H_{K+1}$ be the bottom and top faces defining $A_0$, respectively; there are at most $N+2$ hyperplanes in the collection $\{H_0,\ldots,H_{K+1}\}$. Within each slice between consecutive hyperplanes, up to a thickening by the distance between them, the region $A$ is effectively an $m$-dimensional axes parallel box with (at most) $N$ translates of the unit cube removed, so that, inductively, there is a decomposition of them into at most $(N+1)^{m-1}$ axes parallel boxes. Since there are at most $N+1$ slices, this leaves us with a decomposition into at most $(N+1)^m$ boxes.
\end{proof}

\begin{proof}[Proof of Theorem \ref{thm:other_patches}]
For a type II patch, by Lemma \ref{lem: controlled complements}, the acceptance domain is an axes parallel box, and hence, by Remark \ref{rem:axes parallel}, the proof of Theorem \ref{main:discrepancy} applies, giving the claim. For a type I patch $\PP = \PP_\I(y,r\Omega)$, by Lemma \ref{lem: controlled complements} the acceptance domain $A(\PP)$ is a box with at most $cr^{d-1}$ translates of the unit cube removed. By the above lemma, we may thus decompose $A(\PP)$ into a union of at most $cr^{(d-1)(k-d-1)}$ axes parallel boxes. Applying the proof of Theorem \ref{main:discrepancy} to each of these finishes the proof.
\end{proof}

\subsection{Other search regions}

In Theorem \ref{main:discrepancy} we were investigating the number of occurrences of a fixed patch $\PP$ within the box of points about $y \in Y$ whose lifts have first $d$ coordinates differing from those of $y$ by at most $R$. Instead of taking a box, we may take any general (bounded) shape in $E$. Given a {\bf search region} $A \subseteq E$, we set
\[
\xi_\PP(y,A) \coloneqq \frac{\#\{y' \in Y \mid \PP_\both(y',r\Omega) = \PP \text{ and } \tilde{y}' - \tilde{y} \in A + F_\rho\}}{\#\{y' \in Y \mid \tilde{y}' - \tilde{y} \in A + F_\rho\}}.
\]
Define $A_d$ to be the canonical projection of $A$ to $\R^d \times \{0\}^{k-d}$:
\[
A_d \coloneqq \{x_1 \in \R^d \mid \textrm{there exists } x_2\in F_\rho \textrm{ with }(x_1, x_2)\in A\}.
\]
Set $X_A \coloneqq \Z^d \cap A_d$, and note that the term of the denominator of $\xi_\PP(y,A)$ is precisely $\# X_A$. To estimate the quantity $\xi_\PP(y,A)$, we shall use the following theorem of Laczkovich.

\begin{theorem}[{\cite[Theorem~1.3]{Lac92}}]\label{thm:laczkovich}
 Let $H \subseteq \R^d$ be a region which is a finite union of integer translates of cubes $[-1/2, 1/2)^d$.
 Then there are dyadic cubes $Q_1, \ldots, Q_n$ (i.e. cubes whose side lengths are powers of $2$) such that
 \[
  H = \biggl( \bigcup_{i=1}^l Q_i \biggr) \setminus \biggl( \bigcup_{i=l+1}^n Q_i \biggr),
 \]
 where the $Q_i \cap Q_j = \emptyset$ whenever $i,j \leq l$ or $i,j > l$, and $Q_i \cap Q_j = Q_j$ whenever $i \leq l < j$.
 Furthermore, such a set of dyadic cubes can be chosen so that the number of cubes of side length $2^m$ is
 $\ll|\partial H| 2^{-m(d-1)}$, where $|\partial H|$ is the $(d-1)$-dimensional Hausdorff measure of the boundary of $H$, and the implies constant only depends on the dimension $d$.
\end{theorem}

\begin{theorem}\label{thm:cube complexes}
Consider a cubical cut and project set $Y$ that is typical in the sense of Theorem \ref{main:discrepancy}, for some $\epsilon>0$, and assume that $d\ge 2$. Fix a patch $\PP$, a point $y\in Y$ and a bounded set $A \subseteq E$.  Let $H$ be the cube-complex covering $X_A$, that is, let $H = \bigcup_{n \in X_A} [-1/2, 1/2]^d + n$. Then
\[
 \abs{\xi_\PP(y, A) - \xi_\PP} \le C \frac{|\partial H|}{\#X_A},
\]
where the constant $C$ depends on $E,\epsilon, \PP$ and $\pi$. In the case of type II patches, $C$ is independent of $\PP$.
\end{theorem}
\begin{proof}
Let $(Q_i)$ be as in Theorem \ref{thm:laczkovich}, such that $H$ is a union of the $(Q_i)_{i=1}^l$ minus the $(Q_j)_{j=l+1}^n$.
By a slight abuse of notation, we let $\#Q_i$ denote the number of integer points in $Q_i$.

The fact that the $Q_i$ are either disjoint or included in one another implies that
\[
 \# X_A \xi_{\mathcal{P}} (y, A) = \sum_{i=1}^{l} {\#Q_i} \xi_{\mathcal{P}} (y,Q_i) - \sum_{j=l+1}^{n} {\#Q_j} \xi_{\mathcal{P}} (y,Q_j).
\]
Since $( \sum_i \#Q_i - \sum_j \#Q_j) = \#X_A$, applying the triangle inequality we obtain
\[
 \#X_A | \xi_{\mathcal{P}} (y,A) - \xi_{\mathcal P} | \leq  \sum_{i=1}^n \# Q_i | \xi_{\mathcal P} (y, Q_i) - \xi_{\mathcal P} |.
\]
Each of the $Q_i$'s is a square, and therefore we can apply Theorem~\ref{main:discrepancy} repeatedly for each $Q_i$. For a dyadic $Q_i$ with side length $2^m$, we have that $\# Q_i = 2^{md}$, so we may write
\[
 | \xi_{\mathcal P} (y, Q_i) - \xi_{\mathcal P} | \leq C \frac{ (\log 2^m)^{k+\epsilon}}{2^{md}} = C \frac{(\log 2^m)^{k+\epsilon}}{\# Q_i}.
\]
Furthermore, there are at most $|\partial H| 2^{-m(d-1)}$ boxes $Q_i$ which have length $2^m$. Therefore
\[
 \#X | \xi_{\mathcal{P}} (y, X) - \xi_{\mathcal P} | \leq C \sum_{m=0}^\infty \frac{|\partial H|}{2^{m(d-1)}} \cdot (\log 2^m)^{k+\epsilon}.
\]
If $d \geq 2$, this sum is finite and we get
\[
 | \xi_{\mathcal{P}} (y, X) - \xi_{\mathcal P} | \leq C' \frac{|\partial H|}{\# X_A}.
\]
\end{proof}

\begin{remark}
The above proof does not apply when $d=1$, but in that case any convex set $A \subseteq E$ is an interval $I$, so that Theorem \ref{main:discrepancy} applies directly. In this case, in fact, it follows from Kesten's Theorem \cite{Kesten66} for $k-d = 1$ (see also Grepstad and Lev \cite{GreLev15} for $k-d > 1$) that the discrepancy is bounded, that is, $|I| \cdot |\xi_\PP(y,I) - \xi_\PP| \leq C_\PP$, although the constant may depend on $\PP$.
\end{remark}

Finally, we wish to express an intrinsic and more natural version of these quantities. Again fix a bounded search region $A$ containing the origin, and an equivalence class of patch $\PP$ of either type I or II, and define
\[
\xi'_\PP(y,A) \coloneqq \frac{\#\{y' \in Y \mid \PP_\both(y',r) \textrm{ and } y' \in A+y\}}{\# (Y \cap (A+y))}.
\]
The following theorem shows that for reasonable regions $A$ this quantity does not differ from $\xi_\PP(y,A)$ by too much:

\begin{lemma} Let $Y$ be a cubical cut and project set, $y \in Y$, $\PP$ be a patch of type I or II, and $A \subseteq E$ be a bounded search region containing the origin. Then there exists a constant $\kappa > 0$, depending only on $\pi$, for which
\[
|\xi_\PP(y,A) - \xi'_\PP(y,A)| \leq 2 \frac{|N_\kappa(\partial A)|}{\# X_A}.
\]
\end{lemma}

\begin{proof} Consider the quantities $\# X_A \xi_\PP(y,A)$ and $\#(Y \cap (A+y)) \xi_\PP'(y,A)$. They are given by the number of occurrences of $\PP$ corresponding to points of $X_A$ and to points of $Y \cap (A+y)$, respectively. More precisely, there is a canonical bijection between the points of $X_A$ and the elements of the set
\[
X_A' \coloneqq \{y' \in Y \mid \tilde{y}' - \tilde{y} \in A + F_\rho\} = \{y' \in Y \mid \rho(\tilde{y}') \in A + \rho(\tilde{y})\},
\]
and the quantity $\# X_A \xi_\PP(y,A)$ is precisely the size of the set
\[
\{y' \in Y \mid \PP(y',r\Omega) = \PP \textrm{ and } \rho(\tilde{y}') \in A + \rho(\tilde{y})\}.
\]
So if $y' \in X_A'$ but $y' \notin A+y$, then $\rho(\tilde{y}') \in A + \rho(\tilde{y}) $ but $y' \in A^\textrm{c} + y$. Since $y,y' \in Y$, there exists a $\kappa > 0$ for which this may only be the case if $y' - y \in N_\kappa(\partial A)$. We have a similar statement for the opposite inclusions, and upon restricting to points corresponding to $\PP$, we have the inequalities
\[|\#X_A - (Y \cap (A+y))| \leq |N_\kappa(\partial A)|,\]
and
\[|\# X_A \xi_\PP(y,A) - \#(Y \cap (A+y)) \cdot \xi_\PP'(y,A)| \leq |N_\kappa(\partial A)|.\]
It follows that
\begin{align*}
|\xi_\PP(y,A) - \xi'_\PP(y,A)| &\leq \left | \frac{(\#(Y \cap (A+y)) - \#X_A) \xi_\PP'(y,A)}{\# X_A}\right | + \frac{|N_\kappa(\partial A)|}{\# X_A} \\
&\leq\frac{|N_\kappa(\partial A)|(1+\xi_\PP'(y,A))}{\# X_A} \leq 2 \frac{|N_\kappa(\partial A)|}{\# X_A}.
\end{align*}
\end{proof}
Finally, we deduce the following intrinsic version of Theorem \ref{thm:cube complexes}. Note that the region $A \cap N_\kappa(A^\textrm{c})^\textrm{c}$ in its statement may be interpreted as the set of points of $A$ sufficiently far from $\partial A$.

\begin{corollary} Consider a cubical cut and project set $Y$ that is typical in the sense of Theorem \ref{main:discrepancy} for some $\epsilon>0$, and assume that $d\ge 2$. Then there exists some $\kappa > 0$, depending only on $\pi$, for which, for any patch $\PP$, point $y\in Y$, and bounded set $A \subseteq E$, we have that
\[
 \abs{\xi'_\PP(y, A) - \xi_\PP} \le C \frac{|N_\kappa(\partial A)|}{|A \cap N_\kappa(A^\textrm{c})^\textrm{c}|}.
\]
In particular, for $A$ convex, bounded and containing a neighbourhood of the origin, we have that
\[
 \abs{\xi'_\PP(y, R A) - \xi_\PP} \le C R^{-1}
\]
for sufficiently large $R$. In each formula the constant $C$ depends on $E, \epsilon$, $\pi$ and, for type I patches, $\PP$.
\end{corollary}

\begin{proof}
It follows from Theorem \ref{thm:cube complexes} and the above lemma that
\begin{equation} \label{eq:intrinsic}
\abs{\xi'_\PP(y, A) - \xi_\PP} \leq  C \frac{|\partial H|}{\#X_A} + 2 \frac{|N_\kappa(\partial A)|}{\# X_A}.
\end{equation}
Firstly, we claim that there exists $\kappa > 0$ depending only on the dimension for which $\# X_A \geq |A \cap N_\kappa(A^\textrm{c})^\textrm{c}|$. In what follows, we shall often identify $A$ with $A_d$, and $|A|$ with $|A_d|$, etc. Consider the set $X'_A \subseteq X_A$ of lattice points which are further than distance $\kappa_1$ from $A_d^\textrm{c}$, where $\kappa_1$ is the length of the diagonal of a unit cube in $\R^d$. The cube-complex $Q$ of unit cubes centred at the points of $X_A'$ contains all points of $A_d$ which are further than distance $\kappa_1$ from $A_d^\text{c}$. It follows that $\# X_A \geq \# X_A' = |Q| \geq |A \cap N_{\kappa_1}(A^\textrm{c})^\textrm{c}|$.

Secondly, we claim that there exist $c, \kappa > 0$ depending only on the dimension for which $c|\partial H| \leq |N_\kappa(\partial A)|$. Construct a cube-complex $Q$ by placing cubes of side lengths $1/2$ at the centres of each face of $\partial H$. Since $|\partial H|$ is equal to the number of its faces, and each cube is disjoint, we have that $|Q| = (1/2)^d |\partial H|$. A face of $\partial H$ must be within some $\kappa_2$ of the boundary of $A_d$, so $Q \subseteq N_{\kappa_3}(\partial A_d)$ for some $\kappa_3 > 0$. It follows that $|\partial H| = 2^d |Q| \leq 2^d |N_{\kappa_3}(\partial A)|$.

Inserting this information into (\ref{eq:intrinsic}), we have that
\[
\abs{\xi'_\PP(y, A) - \xi_\PP} \leq  (C'+2) \frac{|N_{\kappa'}(\partial A)|}{|A \cap N_{\kappa'}(A^\textrm{c})^\textrm{c}|}
\]
where $C'$ is independent of $\PP$ in the case of a patch of type II and $\kappa' = \max\{\kappa,\kappa_1,\kappa_3\}$ depends only on $\pi$. For any $\kappa > 0$ and $A$ sufficiently regular  (e.g. convex and containing a neighbourhood of the origin), there exist constants $c_1, c_2$ for which $|N_\kappa(\partial RA)| \leq c_1 R^{d-1}$ and $|RA \cap N_\kappa(RA^\textrm{c})^\textrm{c}| \geq c_2 R^d$ for sufficiently large $R$, from which the claim follows. \end{proof}

\bibliographystyle{abbrv}

\bibliography{./biblio}

\vspace*{.1in}

{\footnotesize

\noindent AH\,:\\
University of Houston, Texas, USA\\
haynes@math.uh.edu

\vspace{.1in}

\noindent AJ\,:\\
Nord University, Levanger, Norway\\
antoine.julien@nord.no

\vspace{.1in}

\noindent HK\,:\\
University of Vienna, Austria\\
henna.koivusalo@univie.ac.at

\vspace{.1in}

\noindent JW\,:\\
University of Durham, UK\\
james.j.walton@durham.ac.uk

}

\end{document}